\documentclass{amsart}
\usepackage{verbatim,amscd, mathabx}

\makeatletter
\def\l@section{\@tocline{1}{0pt}{1pc}{}{}}
\def\l@subsection{\@tocline{2}{0pt}{1pc}{4.6em}{}}
\def\l@subsubsection{\@tocline{3}{0pt}{1pc}{7.6em}{}}
\renewcommand{\tocsection}[3]{%
  \indentlabel{\@ifnotempty{#2}{\makebox[2.3em][l]{%
    \ignorespaces#1 #2.\hfill}}}#3}
\renewcommand{\tocsubsection}[3]{%
  \indentlabel{\@ifnotempty{#2}{\hspace*{2.3em}\makebox[2.3em][l]{%
    \ignorespaces#1 #2.\hfill}}}#3}
\renewcommand{\tocsubsubsection}[3]{%
  \indentlabel{\@ifnotempty{#2}{\hspace*{4.6em}\makebox[3em][l]{%
    \ignorespaces#1 #2.\hfill}}}#3}
\makeatother

\newtheorem{theorem}{Theorem}[section]
\newtheorem{lemma}[theorem]{Lemma}
\newtheorem{proposition}[theorem]{Proposition}

\theoremstyle{definition}

\newtheorem{example}[theorem]{Example}

\theoremstyle{remark}
\newtheorem{remark}[theorem]{Remark}

\numberwithin{equation}{section}

\newcommand{\R}{\mathbb{R}}
\newcommand{\N}{\mathbb{N}}

\newcommand{\X}{\mathrm{X}}

\newcommand{\Y}{\mathrm{Y}}

\newcommand{\1}{\boldsymbol{1}}

\newcommand{\esssup}{\mathop{\rm ess\ sup}}

\newcommand{\vep}{\varepsilon}
\renewcommand{\epsilon}{\vep}

\begin{document}

\title{ODE to $L^p$ norms}

\author{Jarno Talponen}
\address{University of Eastern Finland\\Institute of Mathematics\\Box 111\\FI-80101 Joensuu\\Finland}
\email{talponen@iki.fi}

\date{\today}

\begin{abstract}
In this paper we relate the geometry of Banach spaces to the theory of differential equations, apparently in a new way.
We will construct Banach function space norms arising as weak solutions to ordinary differential equations (ODE)
of the first order. This provides as a special case a new way of defining varying exponent $L^p$ spaces, 
different from the Musielak-Orlicz type approach. We explain heuristically how the definition of the norm by means of the particular ODE is justified.
The resulting  class of spaces includes the classical $L^p$ spaces as a special case. A noteworthy detail regarding our $L^{p(\cdot)}$ norms is that they satisfy H\"older's inequality (properly). 
\end{abstract}

\maketitle


 \section{Introduction}
In this paper we introduce a novel way of defining function space norms by means of weak solutions to ordinary differential equations (ODE). This provides a new perspective for looking at varying exponent $L^p$ spaces. 

The classical Birnbaum-Orlicz  norms were defined in the 1930's, and since then there have been various generalizations of the these norms to several directions. Notable examples of norms and spaces carry names such as Besov,  Lizorkin, Lorentz, Luxemburg, Musielak, Nakano,  Orlicz,  Triebel, Zygmund, see e.g. \cite{bo}, \cite{lux}, \cite{bej}, \cite{mus}. These norms have been recently applied to other areas 
of mathematics as well as to some real-world applications, see e.g.  \cite{die}, \cite{rr2}. Roughly speaking, these norms can be viewed as belonging to a family of derivatives of the Minkowski functional.
This kind of approach leads to several varying exponent $L^{p(\cdot)}$ type constructions, e.g. for sequence spaces, Lebesgue spaces, Hardy spaces and Sobolev spaces. There is a vast literature on these topics, see \cite{kovacik}, \cite{LT}, \cite{nakai} and \cite{rr} for samples and further references. There are also other ways of looking at the varying exponent $L^p$ spaces, such as the Marcinkiewicz space,
whose approach differs from the one mentioned above, see \cite{mar}.

Let us recall that the general Nakano or Musielak-Orlicz type norms are defined as follows:
\[|||f||| =\inf \left\{\lambda>0 \colon \int_\Omega \phi\left(\frac{|f(t)|}{\lambda},t\right)\ dm(t)\leq 1\right\}.\]
Here $\phi$ is a positive function satisfying suitable structural conditions. For instance, $\phi(s,t)=s^{p(t)}$, or 
$\psi(s,t)=\frac{s^{p(t)}}{p(t)}$, $1\leq p(\cdot) <\infty$, produces a norm that can be seen as a varying exponent $L^p$ norm. In the latter case we use the name \emph{Nakano norm} (cf. \cite{JKL}), which, as it turns out, is of particular interest in this paper.

In contrast, the basic form of the norm that will be introduced here differs considerably from the above-mentioned norms in the sense 
that it does not arise as a derivation of the Minkowski functional, and it does not apply any norming set of functionals either. 
In some cases the classes of spaces introduced here do not coincide as sets with any of the classes mentioned above for a given $p\colon [0,1]\to [1,\infty)$ measurable. This is due to obstructions that will become obvious shortly. However, roughly speaking, the norms studied here are equivalent to the 
Nakano norms; see Proposition \ref{prop: upineq}. 

The above norms enjoy the attractive property of being rearrangement invariant in the sense that applying a measure-preserving transformation $T\colon \Omega\to \Omega$ onto
such that $\psi(|f(x)|,x)=\psi(|f\circ T(x)|,T(x))$ for a.e. $x\in \Omega$ (see the above Nakano norm) does not change the value of the norm. However, one may argue that the rearrangement invariance and the apparent simplicity of the definition of the norm come with a cost. Namely, the definition of the norm is opaque in the sense that it involves an infimum with an integral formula inequality having rather complicated interdependencies at the binding surface of the feasible set. For instance, by looking at the definition of the norm it is difficult to decide how adding $1_{\Delta}$, $\Delta\subsetneq \Omega$ measurable, $m(\Delta)>0$, to $f$ contributes to the norm, even if $\Delta$ is in some sense conveniently displaced. Here $1_\Delta$ is the characteristic function of the set $\Delta$.

The `virtues and vices' of the norms about to be introduced are mirror images of the ones mentioned above. The ODE driven norms here, in comparison, will typically not be rearrangement invariant in the above sense, and in particular they do not reduce isometrically to the above Nakano norms (e.g. an example after Proposition \ref{prop: upineq}, cf. an example in \cite{talponen}). On the other hand, our norms will be `localized' in the sense that one can analyze the (infinitesimal) contribution of a single coordinate to the norm so far, a built-in feature of the construction. To make a point, it is possible to compute these norms by solving the defining ODE numerically for continuous functions $f$ and $p$. 
(It is, of course, also possible to compute the above infimum numerically, but we stress the fact that the methods needed to solve our first order ODE are linear in nature and elementary.) 
Thus, our approach to the definition of varying exponent $L^p$ space norms is rather inductive than global. 

Next we will discuss the motivating ideas behind the ODE driven norms. The author studied in \cite{talponen} varying exponent $\ell^{p(\cdot)}$ spaces formed in the following na\"ive fashion. As usual, we denote by $\X \oplus_p \Y$ the direct sum of Banach spaces $\X$ and $\Y$ with the norm given by 
\[\|(x,y)\|_{\X \oplus_p \Y}^p =\|x\|_{\X}^p + \|y\|_{\Y}^p,\quad x\in \X,\ y\in \Y,\ 1\leq p < \infty .\]
Let $p\colon \N \to [1,\infty)$ be a `varying exponent'. Define first a $2$-dimensional Banach space by $\R \oplus_{p(1)} \R$, then a $3$-dimensional one $(\R \oplus_{p(1)} \R)\oplus_{p(2)} \R$ and proceed recursively to obtain $n$-dimensional spaces
\[(\ldots ((\R \oplus_{p(1)} \R)\oplus_{p(2)} \R) \oplus_{p(3)} \ldots )\oplus_{p(n-1)} \R,\]
and, finally, by taking an inverse limit, this yields a space which can be written formally as 
\[\ldots (\ldots ((\R \oplus_{p(1)} \R)\oplus_{p(2)} \R) \oplus_{p(3)} \ldots )\oplus_{p(n)} \R)\oplus_{p(n+1)} \ldots\quad .\]
Thus, this is a space normed by taking a limit of semi-norms corresponding to the $n$-dimensional spaces above. The recursive construction of the spaces can be regarded trivial at each step, but the end result may exhibit some peculiar properties,
depending on the selection of the sequence $(p(n))_{n\in\N}$, see \cite{talponen}. For instance, it provides an \emph{easy} example of a separable Banach space $\X$ with a $1$-unconditional basis such that $\X$ contains all spaces $\ell^p$, $1\leq p<\infty$, almost isometrically. In any case, this appears a rather natural way of constructing Banach sequence spaces and appears to have been first discovered by A. Sobczyk and J.W. Tukey \footnote{The author presented some of the ideas in this paper in the following function space related meetings in 2014: Edwardsville, Albacete, Mekrijarvi; in 2015: 
Kitakyushu, Delhi. The author is grateful to L. Maligranda for bringing the above historical account to author's attention in Kitakyushu (see also \cite{Maligranda}).} (see \cite[p. 96]{Sobczyk}, cf. Kalton et al. \cite{kalton1,kalton2}).

The main aim of this paper is to study `a continuous version' of the above class of sequence spaces $\ell^{p(\cdot)}$, thus a space of suitable functions $f\colon [0,1]\to \R$, instead of sequences. The idea is somewhat similar here, knowing the norm of $f$ up to a coordinate $0<t<1$, 
i.e. $\|1_{[0,t]}f\|$, and knowing the value $|f(t^+ )|$ is sufficient information in predicting the accumulation of the norm right after $t$, i.e. knowing $\|1_{[0,t + dt]}f\|$. 
For example, if $f(r)=0$ for $t< r < s$, then we should have 
$\|1_{[0,t]}f\|=\|1_{[0,s]}f\|$, and if $|f(t^+ )|>0$, then $\|1_{[0,t]}f\|<\|1_{[0,s]}f\|$, and so on.
This intuitive description of the accumulation of the norm is captured by a suitable ODE in such a way that its weak solution, $\varphi_f \colon [0,1]\to [0,\infty)$, shall represent the norm as follows:
\begin{equation}\label{eq: main}
\varphi_f (t)=\|1_{[0,t]}f\|,
\end{equation}
so that in particular $\varphi_ f (0)=0$ and $\varphi_ f (1)=\|f\|$. The above equation \eqref{eq: main} neatly outlines the overall strategy implemented in the beginning of the paper. The basic idea in accomplishing this and the heuristic motivation appear shortly, see Section 1.2. Differential equations have been previously studied in connection to varying exponent spaces and Sobolev spaces (see. e.g. \cite{die0}) but apparently not in the same vein as they arise here.

The required mathematical machinery in this paper is classical, and there is no apparent reason why this alternative approach could have not been experimented with much earlier. 
Also, our approach does not lead to excessively technical considerations, so hopefully it is accessible to a wide range of analysts.

\subsection{Preliminaries and auxiliary results}

We will usually consider the unit interval $[0,1]$ endowed with the Lebesgue measure $m$. 
Here for almost every (a.e.) refers to $m$-a.e., unless otherwise specified. 
Denote by $L^0$ the space of Lebesgue-to-Borel measurable functions on the unit interval.
We denote by $\ell^0 (\N)$ the vector space of sequences of real numbers with point-wise operations.
We refer to \cite{DE_book}, \cite{FA_book},\cite{LT} and \cite{rudin} for suitable background information.

We will mainly study here varying exponent $L^p$ spaces with ODE-determined norm, denoted by $L^{p(\cdot)}$ and 
$\|\cdot\|_{p(\cdot)}$ here, respectively. The author considers these notations intuitive, even though in the literature the Nakano spaces and norms sometimes bear such notations. Therefore, when Nakano norms are considered here, they are explicitly specified and are denoted by $|||\cdot|||$ to clearly distinguish them. 

We will study Carath\'eodory's weak formulation to ODEs, that is, in the sense of Picard type integral formulation, where solutions are required to be only absolutely continuous. This means that, given an ODE 
\[\varphi(0)=x_0 ,\ \varphi'(t)=\Theta(\varphi(t),t),\quad \mathrm{for\ a.e.}\ t\in [0,1],\]
we call $\varphi$ a weak solution in the sense of Carath\'eodory if $\varphi$ is absolutely continuous, $t\mapsto  \Theta(\varphi(t),t)$ is measurable and 
\[\varphi(T)= x_0 + \int_{0}^T \Theta(\varphi(t),t)\ dt\]
holds for all $T\in [0,1]$, where the integral is the Lebesgue integral. In what follows, we will refer to Carath\'eodory's solutions simply as solutions.

Whenever we make a statement about a derivative we implicitly state that it exists. 
We will write $F\leq G$, involving elements of $L^0$, if $F(t)\leq G(t)$ for a.e. $t\in [0,1]$.
We denote the characteristic function or indicator function by $1_A$ defined by $1_A (x)=1$ if $x\in A$
and $1_A (x)=0$ otherwise.

\begin{lemma}\label{lm: second}
Suppose that $\varphi,\psi \in C[0,1]$ are absolutely continuous such that  $\varphi(0)\leq \psi(0)$ and
\[\varphi(t)\geq \psi(t) \Rightarrow \varphi' (t)\leq \psi' (t) \ \text{a.e.}\] 
Then $\varphi\leq \psi$.
\end{lemma}
\begin{proof}
Observe that 
\[\varphi' (t) \leq (\min(\varphi,\psi))' ,\quad \mathrm{for\ a.e.}\ t\in [0,1].\]
\end{proof}

We will frequently calculate terms of the form $(a^p + b^p)^{\frac{1}{p}}$ where $a,b\geq 0$ and 
$1\leq p <\infty$. We will adopt from \cite{talponen} the following short hand notation for this:
\[a \boxplus_p b = (a^p + b^p)^{\frac{1}{p}}.\]
This defines a commutative semi-group on $\R_+$, in particular, the associativity
\[a \boxplus_p (b \boxplus_p c )=(a \boxplus_p b) \boxplus_p c ,\]
is useful.

The space $\ell^{p(\cdot)} \subset \ell^0$, $p\colon \N \to [1,\infty)$, consists of those elements $(x_n)$ such that the 
following limit of a non-decreasing sequence exists and is finite:
\[\lim_{n\to \infty} (\ldots (((|x_1 | \boxplus_{p(1)} |x_2 |)\boxplus_{p(2)} |x_3 |) \boxplus_{p(3)} |x_4 |)\boxplus_{p(4)}\ldots \boxplus_{p(n-1)}|x_n |)\boxplus_{p(n)}|x_{n+1}|\]
and the above limit becomes the norm of the space, see \cite{talponen}.

\subsection{Arriving at the varying exponent $L^p$ norm ODE}

Let us `derive' heuristically our basic differential equation for varying exponent $L^p$ norm. As mentioned in the introduction, we wish to extend the varying exponent $\ell^{p(\cdot)}$ norm in the sense of \cite{talponen} to a continuous setting. Although the motivation for the task here involves the
above  sequence spaces, we are only required to look at simple structures $\X \oplus_p \Y$ one at a time due to the infinitesimal nature of the enterprise.

We will assume a Platonist approach on developing the definition of the varying exponent norms here. Thus we wish to find a function space norm following
 the gist of $\ell^{p(\cdot)}$ space norms. This leads to thought experiments on the right behavior of the 
function $t\mapsto \|1_{[0,t]} f\|$. In a sense, the resulting ODE will be a very robust one, and this allows us to write arguments in this paper 
in a concise fashion, not paying very much attention on the general theory of the ODEs involved. 

Suppose that we have a varying exponent, i.e. a measurable function $p\colon [0,1]\to [1,\infty)$ and $f\colon [0,1]\to \R$ is another measurable function, a possible candidate to lie in the function space. 
We wish to arrange matters in such a way that we have an absolutely continuous non-decreasing function $\varphi_f \colon [0,1]\to [0,\infty)$ such that 
\[\varphi_f (t)=\|1_{[0,t]}f\|, \quad 0\leq t \leq 1,\]
so $\varphi_f (0)=0$ and $\varphi_f (1)=\|f\|<\infty$. 

For example, in the classical case of $L^p$ spaces with a constant function $f=\boldsymbol{1}$ and $p=1,2,\infty$ we have 
\[\varphi_{f,1} (t)=t,\ \varphi_{f,2} (t)=\sqrt{t},\ \mathrm{and}\ \varphi_{f,\infty} (t)=1_{(0,1]}(t),\] 
respectively. Here the $p$-norms are $1$ but the profiles differ
considerably. The first two solutions are absolutely continuous and the last one is not even continuous. 

We will study Carath\'eodory's weak formulation to ODEs. It is convenient to work with absolutely continuous solutions, since this way we may apply such usual tools as Fatou's lemma and Lebesgue's convergence theorems on the solutions (sometimes implicitly). We are only interested here in Banach lattice norms, therefore $\varphi_f$ is always non-decreasing here. In fact, we will require a mildly modified version of Carath\'eodory's weak formulation, tailor-made specifically to our setting.

We are aiming at a recursive-like formula for $\varphi_f$, in a similar spirit as in \cite{talponen}, so suppose that we have defined the function $\varphi_f$ up to the interval $[0,t_0 ]$. Then we are not interested in the values of $f$ and $p$ on $[0,t_0 )$, a Markovian type condition. 
Suppose, as a thought experiment, that $f$ and $p$ are constant on an interval\\ $[t_0 ,t_0 +\Delta]$ where $\Delta>0$. Then we should have 
\begin{equation}\label{eq: phiDelta}
\begin{split}
\varphi(t_0 +\Delta)&=(\varphi(t_0 )^{p(t_0 )}+\Delta |f(t_0 )|^{p(t_0 )})^{1/p(t_0 )},\\
&=\varphi(t_0 ) \boxplus_{p(t_0 )} \Delta^{1/p(t_0 )}|f(t_0 )|
\end{split}
\end{equation}
analogous to the $\ell^{p(\cdot)}$ construction, and actually to the usual $L^p$ norm formula, since 
\begin{multline*}
\left(\int_{0}^{t_0 +\Delta} |f(s)|^p\ dm(s)\right)^{\frac{1}{p}}=\left(\int_{0}^{t_0} |f(s)|^p\ dm(s)\right)^{\frac{1}{p}}\boxplus_p  \left(\int_{t_0 }^{t_0 +\Delta} |f(t_0 )|^p\ dm(s)\right)^{\frac{1}{p}}
\end{multline*}
where the right-most term is $\Delta^{1/p(t_0 )}|f(t_0 )|$.
Thus, by differentiating \eqref{eq: phiDelta} we find a natural candidate for the norm-determining differential equation:
\begin{equation}\label{eq: diff1}
\frac{d^+}{d\Delta}\varphi(t_0 +\Delta)\bigg|_{\Delta=0}=\frac{|f(t_0 )|^{p(t_0 )}}{p(t_0 )}\varphi(t_0 )^{1-p(t_0 )}.
\end{equation}
Here $\frac{d^+}{d\Delta}$ denotes the right-sided derivative and we set $\Delta=0$, because we are interested in `infinitesimal' increments around $t_0$. 
So, the above equation is right if $f$ and $\varphi$ are constant on the interval $[t_0 ,t_0 +\Delta]$, but the equation does \emph{not} concern the values of $f$, $\varphi$ and $p$ \emph{beyond} $t_0$.

In formulating the differential equation we do not require $f$ or $p$ to be continuous anywhere, but motivated by Lusin' s theorem and related considerations we will use the above formula in any case and aim to define
$\varphi$ by 
\begin{equation} \label{eq: def0}
\varphi(0)=0 ,\ \varphi'(t)=\frac{|f(t)|^{p(t)}}{p(t)}\varphi(t)^{1-p(t)}\quad \mathrm{for\ a.e.}\ t\in [0,1].
\end{equation}
Looking at this ODE it becomes evident that if there is a solution $\varphi_f$ corresponding to $f$, then there is also a solution $\varphi_{cf}$ corresponding to $cf$ for any constant $c\in \R$ and moreover the functional 
$f \mapsto \varphi_{f} (1)$ is positively homogenous (up to the uniqueness of the solutions). 

This formulation has the drawback that $0^{1-p(t)}$ is not defined. Also, it has a trivial solution $\varphi \equiv 0$, regardless of the values of $f$, if we use the convention $0^0 =0$ and $p\equiv 1$. Also, following this idea it is possible to construct other degenerate solutions such that $\varphi$ vanishes on $[0,t]$ for any $0< t <1$.
The behavior of the solutions is difficult to anticipate in the case where $\varphi(t)$ is small and $p(t)$ is large.

To fix these issues, we will consider \emph{stabilized} solutions to the above initial value problem. Namely, we will use initial values $\varphi(0)={x_0}>0$ and to correct the error incurred we let ${x_0}\searrow 0$. It turns out that the corresponding unique solutions $\varphi_{{x_0}}$ decreasingly converge point-wise to $\varphi$ which again satisfies the same ODE. So, this procedure yields a unique solution $\varphi$ which we will formulate, by a slight abuse of notation, as
\begin{equation}\label{eq: def}
\varphi(0 )=0^+ ,\ \varphi'(t)=\frac{|f(t)|^{p(t)}}{p(t)}\varphi(t)^{1-p(t)}\quad \mathrm{for\ a.e.}\ t\in [0,1].
\end{equation}
There is more to the above procedure than merely picking a maximal solution; it turns out that in many situations it is convenient to look at positive-initial-value solutions first.
\newcommand{\supp}{\mathrm{supp}}
By using Lebesgue's monotone convergence theorem and Lemma \ref{lm: second} one easily verifies that if for each $x_0 >0$ there is
$\varphi_{x_0}$, a solution to \eqref{eq: def0}, except with initial value $\varphi_{x_0}(0)=x_0$, then 
there is  $\varphi$, a unique solution to \eqref{eq: def0}, such that 
$\varphi_{x_0} \searrow \varphi$ uniformly and $\varphi_{x_0}' \nearrow \varphi'$ in $L^1$ 
as $x_0 \searrow 0$. This unique solution is referred to by \eqref{eq: def}.
Moreover, if such $0^+$-initial value solution exists, then for each positive initial value $x_0 >0$ a corresponding solution exists
by suitable Picard iteration and Lemma \ref{lm: second}.

We define the varying exponent class $L^{p(\cdot)}\subset L^0$ (ODE-determined) as the set of those functions $f\in L^0$ such that  $\varphi_f$ exists as an absolutely continuous solution to \eqref{eq: def} and $\varphi_f (1)<\infty$.
In many cases, but not always, the class becomes a linear space. In such a case the norm can be defined as $f\mapsto \varphi_f (1)$. 

{\noindent \bf Warning 1:} Even if the class $L^{p(\cdot)}$ fails to be a linear space, we sometimes write $\|f\|_{L^{p(\cdot)}} =\|f\|_{p(\cdot)} :=\varphi_f (1)$ where $\varphi_f$ is the solution to \eqref{eq: def}.

{\noindent \bf Warning 2:} As explained above, the class $L^{p(\cdot)}$ and the mapping $f\mapsto \|f\|_{p(\cdot)}$ may differ from the Nakano space and the corresponding norm which are often denoted by the same symbols in the literature. 
      
The above ODE is a separable one for a constant $p(\cdot)\equiv p$, $1\leq p<\infty$, and solving it (see  \eqref{eq: sep}) yields $\varphi_f (1)^{p}=\int_{0}^{1}|f(t)|^p \ dt$, compatible with the classical definition of the $L^p$ norm.
If $p(\cdot)$ is locally bounded and $|f(t)|^{p(t)}$ is locally integrable, then Picard iteration performed locally yields a unique solution for each initial value $\varphi(0)=a>0$, possibly $\varphi(s)\to \infty$ as $s\nearrow  r$ for some $0< r \leq1$. 

\section{Constructions of ODE-determined $L^{p(\cdot)}$ spaces}
In this section we will study only spaces of the type $L^{p(\cdot)}$ with $p\colon [0,1]\to [1,\infty)$ measurable. Some of the unbounded functions $p(\cdot)$ actually produce a class of functions, 
rather than a linear space (see Example \ref{ex: notin}). We will first restrict our considerations to those $L^{p(\cdot)}$ classes which 
are Banach spaces (see Theorem \ref{prop: coincidence} below).
The norms of these spaces were described in the introductory part.

\subsection{Transcending from discrete to continuous state}

We will traverse from varying exponent sequence spaces to such function spaces through an intermediate 
notion which we call simple semi-norm. Let us first we define a very simple semi-norm by the formula
\[|f|_{p, \mu} = \left(\int |f|^p\ d\mu \right)^{\frac{1}{p}}\]
where $\mu$ is a restricted Lebesgue measure with $\mathrm{supp}( \mu )\subset  [0,1]$. Let us consider 
such measures $\mu_i$ with $\max \mathrm{supp}(\mu_i)\leq \min \mathrm{supp}(\mu_{i+1})$, $1\leq i \leq n-1$ and 'exponent constants' $p_i \in [1,\infty)$. Then we may define a composite semi-norm as follows
\begin{multline}\label{eq: Nsemi}
\|f\|_{(\ldots (L^{p_1}(\mu_1 )\oplus_{r_2}L^{p_2}(\mu_2 ))\oplus_{r_3} \ldots  \oplus_{r_{n-1}} L^{p_{n-1}}(\mu_{n-1}))\oplus_{r_{n}} L^{p_{n}}(\mu_n ) }\\
:=(\ldots (|f|_{p_1 , \mu_1}\boxplus_{r_2} |f|_{p_2 , \mu_2}) \boxplus_{r_3} \ldots \boxplus_{r_{n-1}} |f|_{p_{n-1} , \mu_{n-1}})\boxplus_{r_{n}} |f|_{p_{n} ,\mu_{n}}\\
=(\dots (\ldots ((|f|_{p_1 , \mu_1}^{r_2} + |f|_{p_2 , \mu_2}^{r_2})^{\frac{r_3}{r_2}} + 
|f|_{p_3 , \mu_3}^{r_3})^{\frac{r_4}{r_3}} + \ldots |f|_{p_{n-1} , \mu_{n-1}}^{r_{n-1}})^{\frac{r_n}{r_{n-1}}} + |f|_{p_n , \mu_n}^{r_n})^{\frac{1}{r_n}}
\end{multline}
where $\max\mathrm{supp}(\mu_i)\leq \min \mathrm{supp}(\mu_{i+1})$, $r_{i+1}\geq p_{i+1}$.
We will frequently consider simple semi-norms 
\[\|f\|_N = \|f\|_{(\ldots (L^{p_1}(\mu_1 )\oplus_{r_2}L^{p_2}(\mu_2 ))\oplus_{r_3} \ldots )\oplus_{r_n} L^{p_{n}}(\mu_n )}\]
and denote this collection by $\mathcal{N}$. Thus $N$ is a name for a semi-norm, and it is recursively 
defined in \eqref{eq: Nsemi}.
Let us say that a semi-norm is of \emph{standard form} if $r_{i+1} = p_{i+1}$ for all $i\in\{1,\ldots,n-1\}$. 
In this case we may, in a sense, extend the elements $N$ of $\mathcal{N}$ to a $L^{\tilde{p}(\cdot)}$ norm by putting 
$\tilde{p}_N (t)=p_i$ for $t\in \mathrm{supp}(\mu_i )$ and $\tilde{p}_N (t)=1$ otherwise (and this extension is unique). 
If $\bigcup_i \mathrm{supp} (\mu_i )= [0,1]$ then a corresponding standard form norm $N\in \mathcal{N}$ satisfies $\|f\|_N = \|f\|_{\tilde{p}(\cdot)}$ 
by subsequent Lemma \ref{lm: approx}.

Observe that the semi-norms are decreasing on $r$:s and increasing on $p$:s. Consider point-wise intervals $[p_t , r_t]$ as follows: 
$p_t =p_{i+1}$ and $r_t = r_{i+1}$ on the support of $\mu_{i+1}$. We denote by $N \preceq p(\cdot)$ whenever $p(t)\in [p_t , r_t]$ for all $t$ such that the interval is defined.

We may define a partial order on $\mathcal{N}$ by setting $N \preceq M$ if the following conditions hold:
\begin{enumerate}
\item A partition given by the supports of the measures corresponding to $N$ is refined by the supports of the measures corresponding to $M$:
\[\forall \mu_{N,i} \ \ \exists \mu_{M, j^{(i)}_{1}}, \ldots, \mu_{M,j^{(i)}_m} \quad \supp(\mu_{N,i})=\bigcup_{1\leq k \leq m} \supp(\mu_{M,j^{(i)}_{k}}).\]
\item $[p_{M,t} , r_{M,t}]\subset [p_{N,t} , r_{N,t}]$ for each $t$ such that the left hand interval is defined. 
\end{enumerate}
This leads to a definition of a varying exponent $L^p$ norm in a natural way as a limit from below. 

Actually, to simplify considerations we will consider simple semi-norms of the standard form.
For these semi-norms we define $N \leq M$ if the union of the supports of $M$ includes that of $N$ and moreover $\tilde{p}_N \leq \tilde{p}_M$ holds.
This is  again a directed poset. We define $\mathcal{N}_{\leq p(\cdot)}$ to be the collection of simple semi-norms $N$ of standard form such that $\tilde{p}_N \leq p(\cdot)$. The sought after norms can be defined by applying one of the above orders, but here we will concentrate on the latter.

Let us define a functional as follows: 
\begin{equation}\label{eq: rho}
\rho (f) := \limsup_{N} \|f\|_{N}
\end{equation}
where the $\limsup$ is taken along $N  \in \mathcal{N}_{\leq p(\cdot)}$ such that $\tilde{p}_N \to p$
in measure, i.e.
\[\rho (f) = \inf_{K  \in \mathcal{N}_{\leq p(\cdot)}} \ \sup_{N \in  \mathcal{N}_{\leq p(\cdot)},\ K \leq N}\  \|f\|_{N}. \]

By thinking of the basic properties of $\limsup$ and the simple semi-norms, we observe that 
functions $f\in L^0$ with $\rho (f) <\infty$ form a linear space and $\rho$ is a semi-norm on it.
We call this space $\widetilde{L}^{p(\cdot)}$ and it turns out the semi-norm is in fact a norm when we 
identify functions in the usual way, i.e. according to a.e. coincidence.

We will connect the above limiting process of semi-norms to ODEs. In doing this we are required to use initial values for the 
ODEs and, consequently, for semi-norms as well. Although this procedure, strictly speaking, cancels the semi-norm property, we may 
modify the composite semi-norms in such a way that the resulting functions have an initial value in a natural way. 
Namely, we begin the recursive construction by using $(L^1 (\delta_0 ) \oplus_{1} L^{p_1} (\mu_1))$ as the first term, in place of $L^{p_1} (\mu_1)$, where $\delta_0$ is the Dirac's delta probability measure concentrated at $0$. Then the value $|f(0)|$ serves as the `initial value of the semi-norm'.  

\begin{lemma}\label{lm: approx}
Let $\rho(f)<\infty$. Suppose that there are compact subsets $C_i \subset [0,1]$, $1\leq i\leq n$, $\max C_i \leq \min C_{i+1}$ 
such that $p|_{C_{i}}\equiv p_i \in [1,\infty)$. Assume additionally that $f= 1_{\bigcup_i C_i} f$ and $p|_{[0,1]\setminus \bigcup_i C_i}\equiv 1$. 
Then the mapping $\varphi\colon [0,1]\to \R$ given by 
$\varphi(t)=\rho(1_{[0,t]}f)$ is absolutely continuous and satisfies 
\[\varphi(0)=0 ,\ \varphi'(t)=\frac{|f(t)|^{p(t)}}{p(t)}\varphi(t)^{1-p(t)}\quad \mathrm{for\ a.e.}\ t\in [0,1].\]
\end{lemma}

\begin{proof}

First we observe the analogous claim on an interval with a constant $p$ by studying the following differential equation:
\[\varphi(a)=c ,\ \varphi'(t)=\frac{|f(t)|^p}{p}\varphi(t)^{1-p}\quad \mathrm{for\ a.e.}\ t\in [a,b]\subset [0,1].\]

We use the separability of the above differential equation and the absolute continuity of $\varphi$ to obtain 
\begin{equation}\label{eq: sep}
\int_{a}^b p\varphi'(t) \varphi^{p-1}\ dt=\bigg|_{t=a}^b \varphi^p (t)= \int_{a}^b |f(t)|^p \ dt .
\end{equation}
Indeed, we see immediately that $\varphi$ defined in the formulation of the Lemma is absolutely continuous in this special case. 
The above calculation considered in backward order shows also that in the constant $p$ case $\varphi$ arises as a solution to the above 
differential equation on that interval. 

From this we obtain the analogous compact subset $C\subset [0,1]$ case by passing to function of the type $1_C f$. It is clear that the resulting 
$\varphi$ is again absolutely continuous and the derivative is 
\[\varphi' (t)  = \frac{|1_C (t) f(t)|^p}{p}\varphi(t)^{1-p} = 1_C (t)\frac{|f(t)|^p}{p}\varphi(t)^{1-p}\ \text{a.e.}\ .\]

This way we easily see that the simple semi-norm accumulation functions 
\begin{equation}\label{eq: semi}
t\mapsto \|1_{[0,t]}f\|_{(\ldots (L^{p_1}(\mu_1 )\oplus_{p_2}L^{p_2}(\mu_2 ))\oplus_{p_3} \ldots )\oplus_{p_n} L^{p_{n}}(\mu_n )}
\end{equation}
can be seen as solutions to 
\[\varphi(0)=0 ,\ \varphi'(t)=\frac{1_{\bigcup_i \mathrm{supp}(\mu_i )}\ |f(t)|^{p(t)}}{p(t)}\varphi(t)^{1-p(t)}\quad \mathrm{for\ a.e.}\ t\in [0,1]\]
where $p(t)=p_i$ for $t\in C_i =\mathrm{supp}(\mu_i )$ and $\varphi'(t)=0$ for $t\in [0,1]\setminus \bigcup_i C_i$.
Indeed, for $x_i = \max C_i$ in \eqref{eq: semi} we obtain an ODE
\[\varphi(x_i )= \|1_{[0,x_i ]}f\|,\ \varphi'(t)=\frac{|f(t)|^{p(t)}}{p(t)}\varphi(t)^{1-p(t)}\quad \mathrm{for\ a.e.}\ t\in C_{i+1}\]
by induction. Note that the $\sup$ in the $\limsup$ in \eqref{eq: rho} is actually attained in this simple case with $p(\cdot)$ essentially 
piecewise constant on $ \bigcup_i C_i$.
\end{proof}

Given a measurable function $p\colon [0,1]\to [1,\infty)$,  by Lusin's theorem there is for each $\epsilon>0$ a compact set $C\subset [0,1]$ with $m([0,1]\setminus C)<\epsilon$ such that $p|_C$ is continuous, thus uniformly continuous and bounded. 

Thus we can find a sequence of compact subsets $C_m \subset [0,1]$ such as above with $m(C_m )  \to 1$ and by taking finite unions of such sets we may assume that the sequence is increasing. Next we assume for technical reasons that all the semi-norms have a fixed positive initial value component, say $x_0 1_{\{0\}} \in L^1 (\delta_0 )$ with $x_0 >0$.
We may construct by a diagonal argument a sequence $(N_n )$ of simple semi-norms of standard form (but with the added initial value) such that $\tilde{p}_{N_n} \to p$ in measure and for every $f\in L^\infty$ and $m \in \N$ we have 
\begin{multline*}
\frac{d}{dt} \|1_{[0,t] \cap C_m }f\|_{N_n} =  \frac{|1_{C_m}(t) f(t)|^{\tilde{p}_{N_n}(t)}}{\tilde{p}_{N_n}(t)} N_n(1_{[0,t] \cap C_m }f)^{1-\tilde{p}_{N_n}(t)} \\
\to \frac{|1_{C_m}(t) f(t)|^{p(t)}}{p(t)} \limsup_{n\to\infty} N_n (1_{[0,t] \cap C_m }f)^{1-p(t)}
\end{multline*}
in measure (more precisely, in $L^0 (C_m)$) as $n\to\infty$. Using the initial value $a$ ensures that $N_n(1_{[0,t] \cap C_m }f)^{1-p(t)}$
are uniformly bounded on $C_m$ and we are also using above the fact that 
$|f(t)|^{\tilde{p}_{N_n}(t)}$ are uniformly bounded on $C_m$. Note that these observations imply that
$\limsup_{n\to\infty} N_n (1_{[0,t] \cap C_m }f)$ is absolutely continuous. Thus, the above yields that 
$\lim_{n\to\infty} N_n (1_{[0,t] \cap C_m } f)$ exists for each $t$ and is absolutely continuous on $t$. 
In the same vein, thinking  of the definition of $\rho$, still with the same initial value component in all the semi-norms, we observe that $\rho(1_{[0,t] \cap C_m} f) = \lim_{n\to\infty} N_n (1_{[0,t] \cap C_m } f)$.
In particular we observe that $t\mapsto \rho(1_{[0,t] \cap C_m} f)$ is absolutely continuous and 
\begin{equation}\label{eq: deriv}
\frac{d}{dt} \rho(1_{[0,t] \cap C_m} f) =  \frac{|1_{C_m}(t) f(t)|^{p(t)}}{p(t)} \rho(1_{[0,t] \cap C_m }f)^{1-p(t)}
\end{equation}
exists a.e. on $C_m$. Letting the initial value $x_0 \searrow 0$ the above terms increase and converge to a value
for a.e. $t \in C_m$, so that Lebesgue's monotone converge theorem yields a solution to the same ODE with initial value $0$. 
According to $m(\bigcup_m C_m )=1$ and hence the positivity of \eqref{eq: deriv} in a positive measure set yield that $\rho$ is a norm, instead of merely being a semi-norm. 

We will denote subsequently the normed space of functions $f \in L^0$ with $\rho(f)<\infty$ by $\widetilde{L}^{p(\cdot)}$. We will next refine and collect some findings obtained so far.

\begin{proposition}
Given a measurable function $p\colon [0,1]\to [1,\infty)$, the class $\widetilde{L}^{p(\cdot)}$ is a Banach space 
with the usual point-wise linear operations and the corresponding norm $\rho$ defined above.
\end{proposition}
\begin{proof}
We have already established above that $\widetilde{L}^{p(\cdot)}$ endowed with the functional 
$\|\cdot\|_{\widetilde{L}^{p(\cdot)}}:=\rho$ is a normed space. To prove completeness we will pass on to an equivalent norm 
\begin{equation}\label{eq: sup}
\sup_{N\in \mathcal{N}_{\leq p(\cdot)}} \|f\|_N .
\end{equation}
Indeed, clearly $\sup_{N} \|f\|_N \geq \rho(f)$ and by using Lemma \ref{lm: approx}
and Proposition \ref{prop: ineq} we see an opposite inequality using a multiplicative constant from the later result. Therefore these norms are equivalent, and it becomes clear that \eqref{eq: sup} defines a complete norm on $\widetilde{L}^{p(\cdot)}$. Hence $\rho$ is complete as well.
\end{proof}

\begin{theorem}\label{thm: tilde}
Let $f\in L^0$ and $p \colon [0,1]\to [1,\infty)$ measurable. The following conditions are equivalent:
\begin{enumerate}
\item{$f \in L^{p(\cdot )}$,} 
\item{$f \in \widetilde{L}^{p(\cdot )}$ and the mapping $t\mapsto \rho (1_{[0,t]} f)$ is absolutely continuous.}
\end{enumerate} 
Moreover, in both (equivalent) cases we have $\varphi_f (1)= \rho(f)$.
\end{theorem}

\begin{proof}
Assume that $f\in L^0$, $\rho(f)<\infty$ and $t\mapsto \rho (1_{[0,t]} f)$ is absolutely continuous, as in the second condition. To show the first condition, without loss of generality we may assume that the sequence $C_m$ considered previously satisfies that $p|_{C_m}$ and $f|_{C_m}$ are continuous for every $m$. By a diagonal 
argument we may choose a sequence of simple semi-norms $N_n$ of standard form such that 
$N_n (1_{[0,t]} f) \to \rho(1_{[0,t]} f )$ uniform on $t$ and $\tilde{p}_{N_n} \to p$ in measure as $n\to\infty$.
Thus
\[\frac{d}{dt} N_n (1_{[0,t]} f) \to \frac{|f(t)|^{p(t)}}{p(t)} \rho(1_{[0,t]} f)^{1-p(t)}\]
in $L^0 (C_m )$ as $n\to\infty$. Thus, using the assumed absolute continuity and the boundedness of $f$ and $p$ on $C_m$, we obtain 
\begin{multline*}
\rho(1_{[0,T]} f) -\rho(1_{[0,S]} f) = \int_{S}^T  \frac{d}{dt} \rho(1_{[0,t]} f)\ dt
= \sup_m \int_{[S,T]\cap C_m}  \frac{d}{dt} \rho(1_{[0,t]} f)\ dt \\
= \sup_m \int_{[S,T]\cap C_m} \frac{|f(t)|^{p(t)}}{p(t)} \rho(1_{[0,t]} f)^{1-p(t)}\ dt
=\int_{S}^T \frac{|f(t)|^{p(t)}}{p(t)} \rho(1_{[0,t]} f)^{1-p(t)}\ dt .
\end{multline*}
Strictly speaking, we are also required to control the term $\rho(1_{[0,t]} f)^{1-p(t)}$ which need not be bounded. However, this can be made bounded by using a positive initial value and then letting the initial value tend to zero, similarly as before. Now, $\rho(1_{[0,t]} f)$ clearly is the required solution to the norm determining ODE.
The last part of the statement follows immediately.

The other direction becomes apparent later when we investigate estimates for the norms by means of 
differential equations, see Proposition \ref{prop: ineq}.
\end{proof}

\begin{remark}\label{remark1}
According to the previous result the functional $f \mapsto \varphi_f (1) =: \|f\|_{p(\cdot)}$ satisfies the triangle inequality.
\end{remark}

\section{Inequalities}

Given a measurable function $p\colon [0,1] \to (1,\infty)$ defined a.e., we denote its point-wise H\"older conjugate by $p^* \colon [0,1]\to (1,\infty)$ (defined a.e.), that is, 
\[\frac{1}{p(t)} + \frac{1}{p^*(t)} =1 \quad \text{for\ a.e.}\ t \in [0,1].\]

\begin{proposition}[H\"older]\label{prop: Holder}
Suppose that $f\in L^{p (\cdot)}$ and $g\in L^{p^* (\cdot)}$ with $1<p (t) <\infty$ for a.e. $t$. Then they satisfy H\"older's inequality:
\[\int_{0}^1 |f(t) g(t)|\ dt\leq \|f\|_{p (\cdot)}\|g\|_{q^* (\cdot)}.\]
\end{proposition}
Although the function classes need not be linear spaces, we still use the norm notation above (instead of $\varphi_f (1)$ etc.); this is to establish a clear connection to the classical case.

\begin{proof}
By using the H\"older inequality for classical $L^p$ and $\ell^p$ spaces, we obtain by induction an analogous statement for spaces of the type 
\[(\ldots (L^{p_1} (\mu_1 ) \oplus_{p_2} L^{p_2}(\mu_2 ))\oplus_{p_3} \ldots )\oplus_{p_n} L^{p_n} (\mu_n )\] 
considered above. That is, if we write 
$\mu (A) = \sum_{i=1}^n \mu_i (A)$ and $f,g \in L^\infty (\mu)$ we have 
\begin{multline*}
\int |fg|\ d\mu \\
\leq \|f\|_{(\ldots (L^{p_1} (\mu_1 ) \oplus_{p_2} L^{p_2}(\mu_2 ))\oplus_{p_3} \ldots )\oplus_{p_n} L^{p_n} (\mu_n )}\\
\cdot \|g\|_{(\ldots (L^{p_{1}^*} (\mu_1 ) \oplus_{p_{2}^*} L^{p_{2}^*}(\mu_2 ))\oplus_{p_{3}^*} \ldots )\oplus_{p_{n}^*} L^{p_{n}^*} (\mu_n )} .  
\end{multline*}

This inequality passes to the limit by an approximation argument similar as seen previously. Namely, we pick by Lusin's theorem an increasing sequence of compact subset $C_n \subset [0,1]$, $n\in\N$, such that $p|_{C_n}$ and $q|_{C_n}$ are continuous. Note that by the compactness of the subset and the continuity of the exponent we have $1< \min_{t \in C_n}p(t) \leq \max _{t \in C_n}p(t) <\infty$ and similar holds for $q(\cdot)$.
Then, similarly as in the proof of Theorem \ref{thm: tilde}, for any sequence of simple semi-norms of standard form $N_k$ with 
$\tilde{p}_{N_k} \to p(\cdot)$ in measure as $k\to\infty$ we have $N_k (1_{C_m} f ) \to \|1_{C_m } f\|_{p(\cdot)}$ as $k\to\infty$. 
By essentially the same argument we also have that if $N_{k}^*$ are the dual simple semi-norms of standard form, obtained by replacing all the exponents with their respective conjugates, then $\tilde{p}_{N_{k}^*} \to p^* (\cdot)$ in measure as $k\to\infty$ and we have 
$N_{k}^* (1_{C_m} g ) \to \|1_{C_m } g\|_{p^* (\cdot)}$ as $k\to\infty$. This shows that
\[\int_{C_m} |fg|\ dt \leq  \|1_{C_m} f\|_{p(\cdot)}  \|1_{C_m} g\|_{p^* (\cdot)}.\]
Letting $m \to \infty$ yields the claimed inequality for $L^\infty$ functions. Finally, by approximating the given functions 
$f\in L^{p (\cdot)}$ and $g\in L^{p^* (\cdot)}$ with functions of the form $1_{D} f, 1_{D} g \in L^\infty$ in measure and using the absolute continuity of the 
solutions $\varphi_f$ and $\varphi_g$ we obtain the statement.
\end{proof}

Going back to simple semi-norms of standard form, note that if $p(\cdot)\equiv p_1$ on $[0, t_0 )$ and  $p(\cdot)\equiv p_2$ on $[t_0 , 1]$ and $f\in L^{p(\cdot)}$ then 
$\|f\|_{L^{p(\cdot)}}=\|f\|_{L^{p_1}(\mu_1 )\oplus_{p_2}L^{p_2}(\mu_2 )}$ where $\mathrm{supp} (\mu_1 )= [0, t_0  ]$ and 
$\mathrm{supp} (\mu_2 )= [t_0 ,1 ]$ (see Lemma \ref{lm: approx}).

It is easy to see that if $p_2 =1$ then letting $p_1 \nearrow \infty$, $t_0 \searrow 0$ we have $\|\1\|_{p(\cdot)} \nearrow 2$. This is perhaps surprising, 
since always $\|\1\|_p =1$ in the constant $p$ case. We may also alter the above example as follows, letting above 
$f_{t_0}\equiv 1/ t_0$ on $[0,t_0 )$ and $f_{t_0}\equiv  1$ on $[t_0,1]$ with $p_1 = 1$ and $p_2 \nearrow \infty$ and $t_0 \to 0^+$ yields that $\|f_{t_0}\|_{p(\cdot)} \to 1$ whereas $\|f_{t_0}\|_1 \to 2$.

We suspect that the above examples are characteristic in the sense that 
\[\frac{1}{2} \|f\|_1 \leq \|f\|_{p(\cdot)}\leq 2\|f\|_\infty\]
should always hold (so that constant $2$ would be the best possible according to the above examples). We leave this open problem for future research. 

In any case, the above inequalities hold 
with other constants in place of $2$. Namely, suppose that $\varphi_f (t_0 ) =  \|f\|_\infty$. Then 
\[\varphi(t_0 )=\|f\|_\infty ,\ \varphi'(t)=\frac{|f(t)|^{p(t)}}{p(t)}\varphi(t)^{1-p(t)}\quad \mathrm{for\ a.e.}\ t_0 \leq t \leq 1 \]
yields
\[\varphi'(t) \leq \varphi(t),\quad \mathrm{for\ a.e.}\ t_0 \leq t \leq 1 . \]
Observe that $\varphi(1) < y(1)$ where $y$ is the solution to $y' = y$ with $y(0)=\|f\|_\infty$, that is, $y(t)=\|f\|_\infty e^t$. 

Let $a\in (1,2)$ be the solution to $a^a =e$. This satisfies that $\frac{b^x}{x}$ is increasing on $x\geq 1$ for all $b>a$.

\begin{proposition}\label{prop: ineq}
The following inequalities hold whenever defined:
\begin{enumerate}
\item{$\frac{1}{1+a} \|1_{p(\cdot)\geq p}f\|_p  \leq \|1_{p(\cdot)\geq p} f\|_{p(\cdot)}$,}
\item{$\frac{1}{1+ae} \|1_{p_{1}(\cdot)\leq p_{2}(\cdot)}f\|_{p_1 (\cdot)}  \leq \|1_{p_1 (\cdot)\leq p_2 (\cdot)} f\|_{p_2 (\cdot)}$,}
\item{$\|f\|_{p(\cdot)} < e\|f\|_\infty $.}
\end{enumerate}
\end{proposition}
\begin{proof}
The last inequality was already proved, and we will verify the middle inequality which is the most complicated one.

Suppose that $p_1 (\cdot)\leq p_2  (\cdot)$ and $f\in L^{p_1 (\cdot)}$, with $\|f\|_{p_1 (\cdot)} = 1+ae$. 

We wish to exclude the case where $\varphi_{p_2 (\cdot),f}(1) < 1$, so suppose that $\varphi_{p_2 (\cdot),f}(1) \leq 1$. 
Let $[r,1]$ be the maximal interval such that 
$\varphi_{p_2 (\cdot),f}(t) \leq \varphi_{p_1 (\cdot),f}(t)$ (and $\varphi_{p_2 (\cdot),f}(t) \leq 1$) on it. Let $A \subset [0,1]$ be the set where $|f(t)|>a$. 
On $[r,1] \cap A$ we have 
\[\frac{|f(t)|^{p_2 (t)}}{p_2 (t)}\varphi_{p_2 (\cdot),f}(t)^{1-p_2 (t)}\geq \frac{|f(t)|^{p_1 (t)}}{p_1 (t)}\varphi_{p_1 (\cdot),f}(t)^{1-p_1 (t)} .\]
Indeed, in this set we have 
\[\frac{|f(t)|^{p_2 (t)}}{p_2 (t)}\geq  \frac{|f(t)|^{p_1 (t)}}{p_1 (t)}\]
and
\[\varphi_{p_2 (\cdot),f}(t)^{1-p_2 (t)}\geq \varphi_{p_2 (\cdot),f}(t)^{1-p_1 (t)} \geq \varphi_{p_1 (\cdot),f}(t)^{1-p_1 (t)} .\]

Thus 
\begin{multline*}
\varphi_{p_1 (\cdot) ,f}(1) - \varphi_{p_2 (\cdot),f}(1) \\
=\int_{r}^1 \frac{|f(t)|^{p_1 (t)}}{p_1 (t)}\varphi_{p_1 (\cdot),f}(t)^{1-p_1 (t)} - \frac{|f(t)|^{p_2 (t)}}{p_2 (t)}\varphi_{p_2 (\cdot),f}(t)^{1-p_2 (t)}
\ dt\\
=\int_{[r,1]\cap A} \frac{|f(t)|^{p_1 (t)}}{p_1 (t)}\varphi_{p_1 (\cdot),f}(t)^{1-p_1 (t)} - \frac{|f(t)|^{p_2 (t)}}{p_2 (t)}\varphi_{p_2 (\cdot),f}(t)^{1-p_2 (t)}
\ dt\\
+ \int_{[r,1]\setminus A} \frac{|f(t)|^{p_1 (t)}}{p_1 (t)}\varphi_{p_1 (\cdot),f}(t)^{1-p_1 (t)} - \frac{|f(t)|^{p_2 (t)}}{p_2 (t)}\varphi_{p_2 (\cdot),f}(t)^{1-p_2 (t)}
\ dt\\
\leq \int_{[r,1]\setminus A} \frac{|f(t)|^{p_1 (t)}}{p_1 (t)}\varphi_{p_1 (\cdot),f}(t)^{1-p_1 (t)} - \frac{|f(t)|^{p_2 (t)}}{p_2 (t)}\varphi_{p_2 (\cdot),f}(t)^{1-p_2 (t)}
\ dt\\
\leq  \int_{[r,1]\setminus A} \frac{|f(t)|^{p_1 (t)}}{p_1 (t)}\varphi_{p_1 (\cdot),f}(t)^{1-p_1 (t)} \ dt
\leq  \|1_{[r,1]\setminus A} f\|_{p_1 (\cdot)} \leq \|a 1_{[r,1]\setminus A} \|_{p_1 (\cdot)}\\
\leq \|a 1_{[0,1]}\|_{p_1 (\cdot)} \leq e\|a 1_{[0,1]}\|_{\infty}=ae.
\end{multline*}
Thus $\varphi_{p_2 (\cdot),f}(1)\geq (1+ae)-ae =1$. 
\end{proof}

The following fact connects the investigated varying exponent norm to the Nakano $L^{p(\cdot)}$ norms:
\[ |||g|||_{p(\cdot)} := \inf\left\{\lambda >0 \colon \int \frac{1}{p(t)} \left( \frac{|g(t)|}{\lambda} \right)^{p(t)}\ dt \leq 1 \right\}  .\] 

\begin{proposition}\label{prop: upineq}
Let $p\in L^0$, $p(\cdot) \geq 1$, and $f\in L^{p(\cdot)}$ (ODE-determined). Then
\[ |||f|||_{p(\cdot)} \leq \|f\|_{p(\cdot)}\leq  2  |||f|||_{p(\cdot)} .\]
\end{proposition}
\begin{proof}
To prove the left hand estimate it suffices the check that if $\lambda=\|f\|_{p(\cdot)}$, then 
$ \int_{0}^1 \frac{1}{p(t)} \left( \frac{|g(t)|}{\lambda} \right)^{p(t)}\ dt \leq 1$. So, suppose that $0<\varphi_f (1) = \lambda$, then 
\[\varphi_{f}' (t) \geq \frac{|f(t)|^{p(t)}}{p(t)} \lambda^{1-p(t)},\]
(with strict inequality in a set of positive measure if $f\neq 0$), so that 
\[\lambda = \varphi_{f} (1)\geq  \int_{0}^1  \frac{|f(t)|^{p(t)}}{p(t)} \lambda^{1-p(t)}\ dt=\int_{0}^1 \lambda \frac{1}{p(t)} \left(\frac{|f(t)|}{\lambda}\right)^{p(t)}\ dt.\]
This yields 
\[\int \frac{1}{p(t)} \left(\frac{|f(t)|}{\lambda}\right)^{p(t)}\ dt \leq 1.\]

To check the latter inequality, we may restrict to the case $|||f|||_{p(\cdot)}=1$ by the positive homogeneity of the norms. 
If $\|f\|_{p(\cdot)}\leq 1$ then we have the claim, so assume that $0< t_0 < 1$ is such that $\varphi_{f} (t_0 ) = 1$. 
Then 
\[\varphi_{f}' (t) \leq  \frac{|f(t)|^{p(t)}}{p(t)},\quad \mathrm{for\ a.e.}\ t \in [t_0 ,1].\]
Thus
\[\varphi_f (1) \leq 1+ \int_{t_0}^1 \frac{|f(t)|^{p(t)}}{p(t)}\ dt \leq 1 + \int_{0}^1 \frac{|f(t)|^{p(t)}}{p(t)}\ dt = 1 + |||f|||_{p(\cdot)}= 2.\]
\end{proof}

Thus the above Nakano norms are equivalent to the ODE-driven norms considered here. However, these norms  do not coincide in general.
For example, if $p_1 (\cdot)$ is $1$ on $[0,\frac{1}{2})$ and $2$ on $[\frac{1}{2},1]$ and $p_2 (\cdot)$ is defined in the opposite way, then 
$|||f|||_{p_1 (\cdot)} = |||f|||_{p_2 (\cdot)}$, in the case of Nakano norms (or Musielak-Orlicz norms), for any $f$ with $f(s)=f(\frac{1}{2} + s)$ for $0\leq s <\frac{1}{2}$.
The same rearrangement invariance condition does not hold for the investigated $\|\cdot\|_{p (\cdot)}$-norms. Indeed, for the constant function $\boldsymbol{1}$ we have 
\[\|\boldsymbol{1}\|_{p_1 (\cdot)}=\sqrt{\left(\frac{1}{2}\right)^{2} +  \frac{1}{2}} =\frac{\sqrt{3}}{2} \approx 0.866
<1.207 \approx \frac{1}{\sqrt{2}}+\frac{1}{2} = \|\boldsymbol{1}\|_{p_2 (\cdot)} .\]

We can use the above ideas to construct counterexamples as well.
\begin{example}\label{ex: notin}
Let $p\colon [0,1]\to [1,\infty)$ be a measurable function defined by 
\[\frac{1}{p(t)}\left(\frac{2}{3}\right)^{1-p(t)}=\frac{1}{t-\frac{1}{2}}\]
if $t\in (\frac{1}{2},1]$ and $p(t)=1$ on $[0,\frac{1}{2}]$. Then the constant function $f=\boldsymbol{1}$ is \emph{not} in 
$L^{p(\cdot)}$. 
\end{example}
Assuming to the contrary, clearly $\varphi_f (\frac{1}{2})$ would be $\frac{1}{2}$. 
Suppose that $t_0 > \frac{1}{2}$ is a point such that $\varphi_f (t)\leq \frac{2}{3}$ for $\frac{1}{2} \leq t \leq t_0$. Then we should have
\[\frac{2}{3} - \frac{1}{2}\geq  \varphi_f (t_0 ) -\varphi_f \left(\frac{1}{2}\right) = \int_{\frac{1}{2}}^{t_0 }  \varphi_{f}'\ dt 
\geq \int_{\frac{1}{2}}^{t_0 } \frac{1}{t-\frac{1}{2}} \ dt =\infty ,\]
contradicting the assumption that $f$ was in the class, that is, having an absolutely continuous solution 
$\varphi_f$. However, if we should allow initial values $\varphi_f (0) = x_0 \geq \frac{1}{2}$, then we have nice 
corresponding solutions. 

Also note that $1_{[0,\frac{1}{2}]}+1_{[0,1]}\in L^{p(\cdot)}$. This means that 
in general the $L^{p(\cdot)}$ class need not be an ideal as a function class (cf. Banach lattice
theory), i.e. $g\in  L^{p(\cdot)}$, $f\in L^0$, $|f| \leq g$, does \emph{not} imply $f\in  L^{p(\cdot)}$.

In the above example, we have $\left(1_{[0,\frac{1}{2}]}+1_{[0,1]}\right), 1_{[0,\frac{1}{2}]}\in L^{p(\cdot)}$ and 
$\left(1_{[0,\frac{1}{2}]}+1_{[0,1]}\right) - 1_{[0,\frac{1}{2}]} = 1_{[0,1]}\notin L^{p(\cdot)}$. This shows that for some $p(\cdot)$ the class $L^{p(\cdot)}$ fails to be a linear space.
This example is a manifestation of the principle that the higher the value of $\varphi$, the more stable the differential equation becomes, 
ceteris paribus.

\subsection{The essentially bounded exponent case} Let us take a look at the nice case where 
$\overline{p}:=\esssup_t p(t)<\infty$ as it turns out that the corresponding spaces have less pathological properties.

We observed previously that $L^{p(\cdot)}$ classes need not have the ideal property in general.
However, in the case $\overline{p}<\infty$ conditions $g\in L^{p(\cdot)}$, $|f|\leq g$ imply that also
$f \in L^{p(\cdot)}$. This follows immediately from the following observation.

\begin{proposition}\label{prop: domin}
Suppose that $\overline{p}<\infty$, $g \in L^{p(\cdot)}$, $|f|\leq |g|$, and $0< x_0 <1$ is a given initial value. Then $f\in L^{p(\cdot)}$ and 
\[|\varphi_{f, x_0 }' | \leq |\varphi_{g, x_0 }' | \left(\frac{x_0}{\varphi_{g, x_0 }(1)}\right)^{1-\overline{p}}.\]
\end{proposition} 
\begin{proof}
Consider a simple semi-norm $N$ applied to $f$, $g$ and with the above initial value:
$\phi(t)=\|1_{[0,t]} f\|_N$ and $\psi(t)=\|1_{[0,t]}g\|_N$. Clearly $\phi \leq \psi$. Denote by $p(\cdot)$ the corresponding piecewise constant exponent. 
According to Lemma \ref{lm: approx} we may differentiate $\phi$ and $\psi$ a.e. We obtain 
\[\phi' (t) =   \frac{|f(t)|^{p(t)} }{p(t)} \phi^{1-p(t)} (t)\]
and 
\[\psi' (t) = \frac{|g(t)|^{p(t)}}{p(t)} \psi^{1-p(t)} (t)\]
so
\[\frac{\phi' (t)}{\psi' (t)}\leq \left( \frac{\phi (t)}{\psi (t)}\right)^{1-p(t)} \leq \left( \frac{\phi (t)}{\psi (t)}\right)^{1-\overline{p}}
\leq \left( \frac{x_0 }{\psi (1)}\right)^{1-\overline{p}}.\]
The existence of the solution for $f$ follows from subsequent Theorem \ref{prop: coincidence}. 
\end{proof}

In particular we remain within the class if we restrict supports.

\begin{proposition}\label{prop: abs_cont}
Let $\overline{p}<\infty$, $f\in L^{p(\cdot)}$ and $A_n \subset [0,1]$ a sequence of measurable subsets such that $m(A_n) \to 0$ as
$n\to\infty$. Then $\|1_{A_n}f\|_{p(\cdot)}\to 0$ as $n\to\infty$. 
\end{proposition}
\begin{proof}
Fix $\epsilon>0$. We claim that given initial value $x_0 = \epsilon>0$, there is $n_0 \in \N$
such that 
\begin{equation}\label{eq: varphi1A}
\varphi_{1_{A_n}f , x_0} (1) < 2\epsilon,\quad n\geq n_0 .
\end{equation}
This clearly suffices for the statement of the lemma.

The absolute continuity of the solution $\varphi_{f ,x_0}$ implies that 
\[\int_{A_n} \varphi_{f ,x_0}' (t)\ dt \to 0,\quad n\to\infty.\]
Then this observation together with Proposition \ref{prop: domin} yields \eqref{eq: varphi1A}.
\end{proof}

Then $L^\infty \subset  L^{p(\cdot)}$ is dense by the triangle inequality. 
For a general measurable exponent $p\colon [0,1]\to [1,\infty)$ we define a natural Banach subspace 
\[L_{0}^{p(\cdot)} :=\overline{\bigcup_{n\in\N} \left\{1_{p(\cdot )\leq n}\ f\colon f\in \widetilde{L}^{p(\cdot)} \right\} }\subset \widetilde{L}^{p(\cdot)}.\]

\begin{theorem}\label{prop: coincidence}
For a general measurable exponent $p(\cdot)$ the above Banach space satisfies $L_{0}^{p(\cdot)} \subset L^{p(\cdot)}$. 
In particular, in the case $\overline{p}<\infty$, we have $L^{p(\cdot)}=\widetilde{L}^{p(\cdot)}$, and, consequently, $L^{p(\cdot)}$ is a Banach space.
\end{theorem}
\begin{proof}
First we will verify the latter part of the statement, so assume that $\overline{p}<\infty$. Let $f\in \widetilde{L}^{p(\cdot)}$.
Similarly as above, let $D_i \subset [0,1]$, $i\in\N$, be measurable compact subsets such that both
$f|_{D_i}$ and $p|_{D_i}$ are continuous and $m(D_i )\to 1$ as $i\to \infty$. Put 
\[\psi_i (t) := \|1_{[0,t]\cap D_i} f\|_{\widetilde{L}^{p(\cdot)}}\]
where we consider the functions with a joint initial value $0<x_0 <1$. We obtain that $\psi_i$ are absolutely continuous and 
\[\psi_{i}' =  \frac{|1_{D_i} f|^{p(t)}}{p(t)}\psi_{i}^{1-p(t)} (t)\ \text{a.e. }\]

Note that 
\[\psi_i (t) \nearrow \psi (t) := \|1_{[0,t]}f\|_{\widetilde{L}^{p(\cdot)}}\]
for $t\in [0,1]$ as $i\to\infty$ by the absolute continuity of simple semi-norms.

Using the positivity of the initial value $x_0$ and $\overline{p}<\infty$, we obtain by studying the derivatives $\psi_{i}'$ that 
\begin{equation}\label{eq: abs_cont}
\psi(1)^{1-\overline{p}}\int_{r}^t \frac{|f(s)|^{p(s)}}{p(s)}\ ds \leq \psi(t) - \psi(r) \leq x_{0}^{1-\overline{p}}\int_{r}^t \frac{|f(s)|^{p(s)}}{p(s)}\ ds .
\end{equation}
We conclude that $\psi$ is absolutely continuous and $\frac{|f|^{p(\cdot)}}{p(\cdot)} \in L^1$. 

According to Dini's theorem
$\psi_i \to \psi$ converges uniformly on $[0,1]$. Moreover, by using again the positivity of the initial value and $\overline{p}<\infty$
we get that $\psi_{i}^{1-p(t)}(t) \to \psi^{1-p(t)}(t)$ in $L^\infty$-norm as $i\to\infty$. Thus 
\[\psi_{i}' \to  \frac{|f|^{p(t)}}{p(t)}\psi^{1-p(t)}\]
in $L^1$ and hence
\[\psi(T)= x_0 +\int_{0}^T \frac{|f|^{p(t)}}{p(t)}\psi^{1-p(t)},\quad T\in [0,1].\]
This shows that $\psi$ is a solution witnessing the fact that $f\in L^{p(\cdot)}$, since $x_0$ was arbitrary.

To verify the first part of the statement, fix $f\in L_{0}^{p(\cdot)}$ and we aim to show that $f\in L^{p(\cdot)}$,
i.e. that there is a solution $\varphi_f$. Let $f_n = 1_{p(\cdot)\leq n}\ f$ for $n\in\N$. 
Clearly $f_{n} \nearrow f$ a.e. as $n\to\infty$. Similarly as above, it follows from the triangle inequality
that $\varphi_{f_n}\nearrow \phi$ uniformly for a suitable $\phi$. We consider all the solutions with a joint 
positive initial value $x_0 >0$.

For each $k\in \N$ and $\varepsilon>0$ there exist by Egorov's theorem a set 
$D \subset \{t\in [0,1]\colon p(t)\leq k\}$ such that $m(\{t\in [0,1]\colon p(t)\leq k\}\setminus D)<\varepsilon$
and 
\[\frac{|f_n(t)|^{p(t)} }{p(t)} \varphi_{f_n}^{1-p(t)} (t) \to \frac{|f(t)|^{p(t)} }{p(t)} \phi^{1-p(t)} (t)\]
uniformly on $D$ as $n\to\infty$.
Thus 
\[\int_D \frac{|f_n(t)|^{p(t)} }{p(t)} \varphi_{f_n}^{1-p(t)} (t)\ dt \to \int_D \frac{|f(t)|^{p(t)} }{p(t)} \phi^{1-p(t)} (t)\ dt .\]
Since $\varepsilon$ was arbitrary we get by Proposition \ref{prop: abs_cont} that 
\[\int_{p(t)\leq k} \frac{|f_n(t)|^{p(t)} }{p(t)} \varphi_{f_n}^{1-p(t)} (t)\ dt \to \int_{p(t)\leq k}  \frac{|f(t)|^{p(t)} }{p(t)} \phi^{1-p(t)} (t)\ dt \]
for each $k\in\N$. 
Since $\|f- 1_{p(\cdot)\leq n}\  f\|_{\widetilde{L}^{p(\cdot)}}\to 0$, we see that 
\[ \lim_{k\to\infty} \int_{p(t)\geq k} \frac{|f(t)|^{p(t)} }{p(t)} = 0 .\]

It follows that 
\[\int_{0}^T \frac{|f_n(t)|^{p(t)} }{p(t)} \varphi_{f_n}^{1-p(t)} (t)\ dt \to \int_{0}^T \frac{|f(t)|^{p(t)} }{p(t)} \phi^{1-p(t)} (t)\ dt,\quad T\in [0,1]. \]
Taking into account that $\varphi_{f_n}\to \phi$ uniformly, we see that
\[\phi(T)= x_0 + \int_{0}^T \frac{|f(t)|^{p(t)} }{p(t)} \phi^{1-p(t)} (t)\ dt,\quad T\in [0,1]. \]
\end{proof}

The above argument (recall \eqref{eq: abs_cont})  yields the following fact.
\begin{proposition}
If $\overline{p}<\infty$, then $f \in L^0$ is in $L^{p(\cdot)}$ if and only if 
\[\int_{0}^1 \frac{|f(t)|^{p(t)}}{p(t)}\ dt <\infty.\]
\end{proposition}
\qed

The $L^{p(\cdot)}$ space construction here can be generalized to a multidimensional setting 
$L^{p(\cdot)} (\Omega)$ with domains $\Omega \subset \R^n$, $n>1$. 
There appear to be several ways to accomplish this. For example, in some cases $\Omega$ can be 
conveniently decomposed to level sets of $p(\cdot)$, then taking the $L^p$ norms relative to each level set and using the approach here to aggregate the $L^p$ norms yields a 
$L^{p(\cdot)} (\Omega)$ norm. We leave this for future research.

\subsection*{Acknowledgments}

This work has been financially supported by research grants from the V\"ais\"al\"a foundation and the Finnish Cultural Foundation 
and the Academy of Finland Project \# 268009.


\begin{thebibliography}{DGZ}
\bibitem{bo}
Z. Birnbaum; W. Orlicz,  \"Uber die Verallgemeinerung des Begriffes der zueinander Konjugierten Potenzen, 
Studia Math. 3 (1931), 1--67.
\bibitem{kalton1} G. Androulakis, C. Cazacu, N. Kalton, Twisted sums, Fenchel-Orlicz spaces and property (M). Houston J. Math. 24 (1998), 105--126.
\bibitem{DE_book}  E. Coddington, N. Levinson, Theory of ordinary differential equations. McGraw-Hill Book Company, Inc., New York-Toronto-London, 1955.
\bibitem{die0}
L. Diening, M. Ruzicka, Calderon-Zygmund operators on generalized Lebesgue spaces $L^{p(\cdot)}$ and problems related to fluid dynamics.
J. Reine Angew. Math. 563 (2003), 197--220.
\bibitem{die}
L. Diening, P. H\"ast\"o, S. Roudenko, 
Function spaces of variable smoothness and integrability, J. Funct. Anal. 256 (2009), 1731--1768.
\bibitem{FA_book} M. Fabian, P.  Habala, P. H\'ajek, V. Montesinos Santalucía, J. Pelant, V. Zizler, 
Functional analysis and infinite-dimensional geometry. CMS Books in Mathematics/Ouvrages de Mathématiques de la SMC, 8. Springer-Verlag, New York, 2001.
\bibitem{JKL}
J.E. Jamison, A. Kami\'nska, P.-K. Lin, Isometries of Musielak-Orlicz spaces II, Studia Math. 104 (1993), 75--89.
\bibitem{kalton2} N. Kalton, Extension of linear operators and Lipschitz maps into C(K)-spaces. New York J. Math. 13 (2007), 317--381.
\bibitem{kovacik} O. Kov\'a\u{c}ik, J. R\'akosn\'ik, On spaces $L^{p(x)}$ and $W^{k,p(x)}$. Czechoslovak Math. J. 41 (1991), 592--618.
\bibitem{LT} J. Lindenstrauss, L. Tzafriri, Classical Banach Spaces (Classics in Mathematics) Paperback, Springer.
\bibitem{lux}
W. Luxemburg, Banach function spaces, T.U. Delft (1955) (Thesis).
\bibitem{Maligranda}
L. Maligranda, Hidegoro Nakano (1909-1974) - on the centenary of his birth.
In: Banach and Function Spaces III, Proceedings of the Third International Symposium on Banach and Function Spaces 2009, Sep. 14-17, 2009, Kitakyushu-Japan (Editors M. Kato, L. Maligranda and T. Suzuki), Yokohama Publishers 2011.
\bibitem{mar}
J. Marcinkiewicz, Sur l'interpolation d'op\'erations C.R. Acad. Sci. Paris, 208 (1939) 1272--1273
\bibitem{mus}
J. Musielak, Orlicz spaces and modular spaces, Lecture Notes in Mathematics, Vol. 1034 Springer-Verlag, Berlin/New York (1983).
\bibitem{nakai} E. Nakai, Y. Sawano, Hardy spaces with variable exponents and generalized Campanato spaces. J. Funct. Anal. 262 (2012), 3665--3748.
\bibitem{rr} M. Rao, Z. Ren,
Theory of Orlicz spaces, Monographs and Textbooks in Pure and Applied Mathematics, vol. 146 Marcel Dekker, Inc., New York (1991).
\bibitem{rr2} M. Rao, Z. Ren, Applications of Orlicz Spaces, CRC Press, 2002.
\bibitem{rudin}
W. Rudin, Real and Complex Analysis, McGraw-Hill 1987.
\bibitem{Sobczyk}
A. Sobczyk, Projections in Minkowski and Banach spaces, Duke Math. J. 8 (1941), 78--106.
\bibitem{talponen} J. Talponen, A natural class of sequential Banach spaces. Bull. Pol. Acad. Sci. Math. 59 (2011), 185--196.
\bibitem{bej}
B. Yaacov, Modular functionals and perturbations of Nakano spaces, Logic and Analysis 1 (2009), 1--42.
\end{thebibliography}
\end{document}